\documentclass[11pt, journal, draftclsnofoot, onecolumn]{IEEEtran}
\ifCLASSOPTIONcompsoc
\else
\fi

\ifCLASSINFOpdf
\else
\fi
\usepackage[parfill]{parskip}    % Activate to begin paragraphs with an empty line rather than an indent
\usepackage{amsmath,amssymb,latexsym,enumerate, mathrsfs}
\usepackage{graphicx}
\usepackage{graphicx}
\usepackage{array}
\usepackage{times}
\usepackage{hyperref}
\newtheorem{theorem}{\bf Theorem}

\newtheorem{corollary}{\bf Corollary}
\newtheorem{problem}{\bf Problem}
\newtheorem{remark}{Remark}
\newtheorem{definition}{Definition}
\newtheorem{proposition}{Proposition}

\def\qed{\hfill $\Box$}
\def\etal{\mbox{et al.}}

\begin{document}
\title{Relating maximum entropy, resilient behavior and game-theoretic equilibrium feedback operators in multi-channel systems}

\author{Getachew~K.~Befekadu~\IEEEmembership{}% <-this % stops a space
and~Panos~J.~Antsaklis
\thanks{This work was supported in part by the National Science Foundation under Grant No. CNS-1035655. The first author acknowledges support from the College of Engineering, University of Notre Dame.}
\IEEEcompsocitemizethanks{\IEEEcompsocthanksitem G. K. Befekadu is with the Department
of Electrical Engineering, University of Notre Dame, Notre Dame, IN 46556, USA.\protect\\
E-mail: gbefekadu1@nd.edu
\IEEEcompsocthanksitem P. J. Antsaklis is with the Department
of Electrical Engineering, University of Notre Dame, Notre Dame, IN 46556, USA.\protect\\
E-mail: antsaklis.1@nd.edu\protect\\
Version - December 19, 2013.}}

\markboth{}%
{Shell \MakeLowercase{\textit{et al.}}: Bare Advanced Demo of IEEEtran.cls for Journals}
% PRIOR to the title within the \IEEEcompsoctitleabstractindextext IEEEtran
\IEEEcompsoctitleabstractindextext{%
\begin{abstract}
%\boldmath
In this paper, we first draw a connection between the existence of a stationary density function (which corresponds to an equilibrium state in the sense of statistical mechanics) and a set of feedback operators in a multi-channel system that strategically interacts in a game-theoretic framework. In particular, we show that there exists a set of (game-theoretic) equilibrium feedback operators such that the composition of the multi-channel system with this set of equilibrium feedback operators, when described by density functions, will evolve towards an equilibrium state in such a way that the entropy of the whole system is maximized. As a result of this, we are led to study, by a means of a stationary density function (i.e., a common fixed-point) for a family of Frobenius-Perron operators, how the dynamics of the system together with the equilibrium feedback operators determine the evolution of the density functions, and how this information translates into the maximum entropy behavior of the system. Later, we use such results to examine the resilient behavior of this set of equilibrium feedback operators, when there is a small random perturbation in the system.
\end{abstract}
% not want either math or citations to appear in the abstract.

}

\maketitle

% is not selected <OR> if conference mode is selected - because compsoc
% conference papers position the abstract like regular (non-compsoc)
\IEEEdisplaynotcompsoctitleabstractindextext

% \fi
\IEEEpeerreviewmaketitle

\section{Introduction} \label{S1}
The main purpose of this paper is to draw a connection between the existence of a stationary density (that corresponds to an equilibrium state in the sense of statistical mechanics) and a set of feedback operators in a multi-channel system that interacts strategically in a game-theoretic framework. We first specify a game in a strategic form over an infinite-horizon -- where, in the course of the game, each feedback operator generates automatically a feedback control in response to the action of other feedback operators through the system (i.e., using the current state-information of the system) and, similarly, any number of feedback operators can decide on to play their feedback strategies simultaneously. However, each of these feedback operators are expected to respond in some sense of best-response correspondence to the strategies of the other feedback operators in the system. In such a scenario, it is well known that the notion of (game-theoretic) equilibrium will offer a suitable framework to study or characterize the robust property of all equilibrium solutions under a family of information structures -- since no one can improve his payoff by deviating unilaterally from this strategy once the equilibrium strategy is attained (e.g., see \cite{Aub93}, \cite{Nas51} or \cite{Ros65} on the notions of optimums and strategic equilibria in games).\footnote{In this paper, we consider this set of feedback operators as noncooperative agents (or players), but fully-rational entities, {\em over an infinite-horizon}, in a game-theoretic sense. Further, {\em at each instant-time}, each feedback operator knows that the others will look for feedback strategies, but they are not necessarily informed about each others strategies.}

In view of the above arguments, we present in this paper an extension of game-theoretic formalism for multi-channel systems that tend to move towards an equilibrium or ``maximum entropy" state in the sense of statistical mechanics (e.g., see Lanford \cite[pp\,77--95]{Lan73} or Ruelle \cite{Rue78}) -- when the criterion is to minimize the relative entropy between any two density functions, for large-time, with respect to the control channels or the class of admissible control functions (i.e., the set of feedback operators). This further allows us to establish a connection between the existence of a stationary density function (which corresponds to a unique equilibrium state) and a set of feedback operators that strategically interacts in the system. 

Moreover, based on a common fixed-point for a family of Frobenius-Perron operators, we provide a sufficient condition on the existence of a set of (game-theoretic) equilibrium feedback operators such that when the composition of the multi-channel system with this set of equilibrium feedback operators, described by density functions, evolves towards an equilibrium state in such a way that the entropy of the whole system is maximized. As a result of this, we are led to study, how the dynamics of the system together with these equilibrium feedback operators determine the evolution of the density functions, and how this information translates into the maximum entropy behavior of the system. Later, we use such results to examine the resilient behavior of this set of equilibrium feedback operators, when there is a random perturbation in the system. We, in particular, establish sufficient conditions, based on the convergence of invariant measures (i.e., stochastic stability -- in the sense of deterministic limit (e.g., see \cite{You86}, \cite{BalYo93}, \cite{Kif88} or \cite{FreWe12} for related discussions), that will guarantee the resilient behavior for the set of equilibrium feedback operators with respect to random perturbations in the system. 

Here, we hasten to add that such a study, which involves evidence of systems exhibiting resilient behavior, would undoubtedly provide a better understanding of reliability or prescribing an optimal (sub-optimal) degree of redundancy in decentralized control systems. Finally, in the information theoretic-games, we also note that the notions of entropy, game-theoretical equilibrium and complexity, based on the {\em Maximum Entropy Principle} (MaxEnt) of Jaynes \cite{Jay57a}, \cite{Jay57b} and the {\em $I$-\,divergence metric} of Csisz\'{a}r \cite{Csi67}, \cite{Csi75}, have been investigated in the context of zero-sum games by Tops\o e (e.g., see \cite{Top93} or \cite{Top04}), Gr\"{u}nwald and Dawid \cite{GruDa04} and, similarly, by Haussler \cite{Hau97}. Moreover, we observe that the notion of entropy (and its variants) has been well discussed in systems theory literature in the context of robustness analysis and/or synthesis for systems with uncertainties (e.g., see \cite{PetJaDu00} and \cite{ChaKyRe06}).

The remainder of the paper is organized as follows. In Section~\ref{S2}, we recall the necessary background and present some preliminary results that are relevant to our paper. Section~\ref{S3} introduces a family of mappings for multi-channel systems that will be used for our main results. In Section~\ref{S4} we present our main results -- where we establish a three way connection between the existence of an equilibrium state (i.e., the maximum entropy in the sense of statistical mechanics), a common stationary density function for the family of Frobenius-Perron operators, and a set of (game-theoretic) equilibrium feedback operators. This section also discusses an extension of the resilient behavior (to these equilibrium feedback operators), when there is a small random perturbation in the system.

\section{Background, Definitions, and Notations} \label{S2}

In the following, we provide the necessary background and recall some known results from measure theory that will be useful in the sequel. The results are standard (and will be stated without proof); and they can be found in standard graduate books (e.g., see \cite{Hal74}, \cite{GikSko75} and \cite{Yos95} on the measure theory; and see also \cite{LasMac94} or \cite{Kif88} on the stochastic aspects of dynamical systems).
\begin{definition} \label{DFN1}
Let $(X, \mathscr{A}, \mu)$ be a measure space and $L^1(X, \mathscr{A}, \mu)$ be the space of all possible real-valued measurable functions $\vartheta\colon X \rightarrow \mathbb{R}$ satisfying
\begin{align}
 \int_{X} \vert \vartheta(x)\vert \mu(dx) < \infty. \label{EQ1}
\end{align}
If $S\colon X \rightarrow X$ is a measurable nonsingular transformation, i.e., $\mu(S^{-1}(A))=0$ for all $A \in \mathscr{A}$ such that $\mu(A)=0$, then the operator $P \colon L^1(X, \mathscr{A}, \mu) \rightarrow L^1(X, \mathscr{A}, \mu)$ defined by
\begin{align}
 \int_{A} P \vartheta(x) \mu(dx) = \int_{S^{-1}(A)} \vartheta(x) \mu(dx), ~~ \forall A \in \mathscr{A}, \label{EQ2}
 \end{align}
is called the Frobenius-Perron operator with respect to $S$.
\end{definition}

\begin{definition}\label{DFN2}
Let $(X, \mathscr{A}, \mu)$ be a measure space. Define
\begin{align}
 D(X, \mathscr{A}, \mu) = \Bigl \{ \vartheta(x) \in L^1(X, \mathscr{A}, \mu) \, \Bigl \vert \,\vartheta(x) \ge 0  ~~ \text{and} ~~ \bigl\lVert \vartheta(x) \bigr\rVert_{L^1(X, \mathscr{A}, \mu)}=1 \Bigr\}. \label{EQ3}
\end{align}
Then, any continuous function $\vartheta(x) \in D(X, \mathscr{A}, \mu)$ is called a density function.
\end{definition}

\begin{definition} \label{DFN3}
Let $(X, \mathscr{A}, \mu)$ be a measure space. If $S\colon X \rightarrow X$ is a nonsingular transformation and $\zeta(x) \in L^{\infty}(X, \mathscr{A}, \mu)$. Then, the operator $U \colon L^{\infty}(X, \mathscr{A}, \mu) \rightarrow L^{\infty}(X, \mathscr{A}, \mu)$ defined by
\begin{align}
 U \zeta(x) = \zeta\bigl(S(x)\bigr), \label{EQ4}
\end{align}
is called the Koopman operator with respect to $S$.
\end{definition}
Note that for every $\zeta(x) \in L^{\infty}(X, \mathscr{A}, \mu)$
\begin{align}
 \bigl \lVert U \zeta(x) \bigr \lVert_{L^{\infty}(X, \mathscr{A}, \mu)} \le \bigl \lVert \zeta(x) \bigr \lVert_{L^{\infty}(X, \mathscr{A}, \mu)}. \label{EQ5}
\end{align}
Moreover, for every $\vartheta(x) \in L^1(X, \mathscr{A}, \mu)$ and $\zeta(x) \in L^{\infty}(X, \mathscr{A}, \mu)$, then we have
\begin{align}
\langle P \vartheta,\, \zeta \rangle = \langle \vartheta,\, U \zeta \rangle, \label{EQ6} 
\end{align}
so that the operator $U$ is an adjoint to the Frobenius-Perron operator $P$.\footnote{$\langle P \vartheta,\, \zeta \rangle \triangleq \int_{X} P\vartheta(x)\zeta(x) \mu(dx)$.}

\begin{remark}
The transformation $S$ is said to be measure preserving if $\mu\bigl(S^{-1}(A)\bigr)=\mu\bigl(A\bigr)$ for all $A \in \mathscr{A}$. Note that the property of measure preserving depends both on $S$ and $\mu$.
\end{remark}

\begin{definition} \label{DFN4}
Let $\vartheta(x) \in L^1(X, \mathscr{A}, \mu)$ and $\vartheta(x) \ge 0$. If the measure
\begin{align}
\mu_{\vartheta}(A) = \int_{A} \vartheta(x) \mu\bigr(dx\bigl), \label{EQ7}  
\end{align}
is absolutely continuous with respect to the measure $\mu$, then $\vartheta(x)$ is called the Radon-Nikodym derivative of $\mu_{\vartheta}$ with respect to $\mu$.
\end{definition}

\begin{theorem} \label{TH1}
Let $(X, \mathscr{A}, \mu)$ be a measure space, $S\colon X \rightarrow X$ be a nonsingular transformation, and let $P$ be the Frobenius-Perron operator with respect to $S$. Consider a nonnegative function $\vartheta(x) \in L^1(X, \mathscr{A}, \mu)$, i.e., $\vartheta(x) > 0$, $\forall x \in X$. Then, a measure $\mu_{\vartheta}$ given by
\begin{align}
 \mu_{\vartheta}(A) = \int_{A} \vartheta(x) \mu(dx), \quad \forall A \in \mathscr{A}, \label{EQ8}
\end{align}
is invariant, if and only if, $\vartheta(x)$ is a stationary density function (i.e., a fixed-point) of $P$.
\end{theorem}

\begin{theorem} \label{TH2}
Let $(X, \mathscr{A}, \mu)$ be a measure space and $S\colon X \rightarrow X$ be a nonsingular transformation. $S$ is ergodic, if and only if, for every measurable function $\vartheta \colon X \rightarrow \mathbb{R}$
\begin{align}
 \vartheta(S(x)) = \vartheta(x), \label{EQ9}
\end{align}
for almost all $x \in X$, implies that $\vartheta(x)$ is constant almost everywhere.
\end{theorem}

\begin{definition} \label{DFN5} Convergence of sequences of functions (e.g., see \cite{Kre85} or \cite{DunSch58}).
\begin{enumerate} [(i)]
\item A sequence of functions $\bigl\{\vartheta_n(x)\bigl\}$, $\vartheta_n(x) \in L^1(X, \mathscr{A}, \mu)$, is {\em weakly Ces\`{a}ro convergent} to $\vartheta(x) \in L^1(X, \mathscr{A}, \mu)$ if 
\begin{align}
\lim_{n \rightarrow \infty} \frac{1}{n} \sum_{k=1}^n \langle \vartheta_n, \,\zeta \rangle = \langle \vartheta, \,\zeta \rangle, \quad \forall \zeta \in L^{\infty}(X, \mathscr{A}, \mu). \label{EQ10}
\end{align}
\item A sequence of functions $\bigl\{\vartheta_n(x)\bigl\}$, $\vartheta_n(x) \in L^1(X, \mathscr{A}, \mu)$, is {\em weakly convergent} to $\vartheta(x) \in L^1(X, \mathscr{A}, \mu)$ if 
\begin{align}
\lim_{n \rightarrow \infty} \langle \vartheta_n, \,\zeta \rangle = \langle \vartheta, \,\zeta \rangle, \quad \forall \zeta \in L^{\infty}(X, \mathscr{A}, \mu). \label{EQ11}
\end{align}
\item A sequence of functions $\bigl\{\vartheta_n(x)\bigl\}$, $\vartheta_n(x) \in L^1(X, \mathscr{A}, \mu)$, is {\em strongly convergent} to $\vartheta(x) \in L^1(X, \mathscr{A}, \mu)$ if 
\begin{align}
\lim_{n \rightarrow \infty} \bigl\Vert \vartheta_n(x) - \vartheta(x) \bigr\Vert_{L^1(X, \mathscr{A}, \mu)} = 0. \label{EQ12}
\end{align}
\end{enumerate}
\end{definition}

\begin{theorem} [Chebyshev's inequality] \label{TH3}
Let $(X, \mathscr{A}, \mu)$ be a measure space and let $V \colon X \rightarrow \mathbb{R}_{+}$ be an arbitrary nonnegative measurable function. Define 
\begin{align}
\mathit{E}\bigl(V(x) \bigl\vert \vartheta(x)\bigr) = \int_{X} V(x) \vartheta(x) \mu(dx), ~~ \forall \vartheta(x) \in D(X, \mathscr{A}, \mu). \label{EQ13}
\end{align}
If \, $G_{\alpha} = \bigl\{ x \in X \,\bigl\vert\, V(x) < \alpha \bigr\}$, then 
\begin{align}
\int_{G_{\alpha}} V(x) \mu(dx) = 1 - \mathit{E}\bigl(V(x) \bigl\vert \vartheta(x)\bigr). \label{EQ14}
\end{align}
\end{theorem}

\begin{theorem} \label{TH4}
Let $(X, \mathscr{A}, \mu)$ be a finite measure space (i.e., $\mu(X) < \infty$) and let $S\colon X \rightarrow X$ be a measure preserving and ergodic. Then, for any integrable function $\vartheta_{\ast}(x)$, the average of $\vartheta(x)$ along the trajectory of $S$ is equal to almost everywhere to the average of $\vartheta(x)$ over the space $X$, i.e.,\footnote{{\bf Theorem} {\bf (Birkhoff's individual ergodic theorem)} Let $(X, \mathscr{A}, \mu)$ be a measure space, $S \colon X \rightarrow X$ be a measurable transformation, and let $\vartheta \colon X \rightarrow \mathbb{R}$ be an integrable function. If the measure $\mu$ is invariant, then there exists an integrable function $\vartheta_{\ast}(x)$ such that
\begin{align*}
 \vartheta_{\ast}(x) = \lim_{n \rightarrow \infty} \frac{1}{n}\sum_{k=1}^{n-1} \vartheta(S^{k}(x)),
\end{align*}
for almost all $x \in X$.}
\begin{align}
 \lim_{n \rightarrow \infty} \frac{1}{n}\sum_{k=1}^{n-1} \vartheta(S^{k}(x)) = \frac{1}{\mu(X)} \int_{X} \vartheta(x)\mu(dx). \label{EQ15}
\end{align}
\end{theorem}

\begin{corollary} \label{CR1}
Let $(X, \mathscr{A}, \mu)$ be a finite measure space and $S\colon X \rightarrow X$ be measure preserving and ergodic. Then, for any set $A \in \mathscr{A}$, $\mu(A) > 0$, and for almost all $x \in X$, the fraction of the points $\bigl\{S^k(x) \bigr\}$ in $A \in \mathscr{A}$ as $k \rightarrow \infty$ is given by $\mu(A)/\mu(X)$.
\end{corollary}

\begin{corollary} \label{CR2}
Let $(X, \mathscr{A}, \mu)$ be a normalized measure space (i.e., $\mu(X)=1$) and let $S \colon X \rightarrow X$ be measure preserving. Suppose that $P$ is the Frobenius-Perron with respect to $S$. Then, $S$ is ergodic if and only if 
\begin{align}
\lim_{n \rightarrow \infty} \frac{1}{n} \sum_{k=1}^{n-1} P^n \vartheta(x) = 1,  \label{EQ16}
\end{align}
for every $\vartheta(x) \in D(X, \mathscr{A}, \mu)$.
\end{corollary}

\begin{theorem} \label{TH5}
Let $(X, \mathscr{A}, \mu)$ be a measure space, $S\colon X \rightarrow X$ be a nonsingular transformation, and let $P$ be the Frobenius-Perron operator with respect to $S$. If $S$ is ergodic, then there exist at most one stationary density function $\vartheta_{\ast}(x)$ of $P$ (i.e., a fixed-point of $P\vartheta_{\ast}(x)=\vartheta_{\ast}(x)$). Furthermore, if there is a unique stationary density function $\vartheta_{\ast}(x)$ of $P$ and $\vartheta_{\ast}(x)>0$ almost everywhere, then $S$ is ergodic.
\end{theorem}

\begin{corollary} \label{CR3}
Let $(X, \mathscr{A}, \mu)$ be a measure space, $S \colon X \rightarrow X$ be a nonsingular transformation, and let $P \colon L^1(X, \mathscr{A}, \mu) \rightarrow L^1(X, \mathscr{A}, \mu)$ be the Frobenius-Perron operator $P$ with respect to $S$. If, for some $\vartheta(x) \in D(X, \mathscr{A}, \mu)$, there is an $\ell(x) \in D(X, \mathscr{A}, \mu)$ such that 
\begin{align}
P^n\vartheta(x) \le \ell(x), ~~ \forall n \ge 0.  \label{EQ17}
\end{align}
Then, there is a stationary density function $\vartheta_{\ast}(x) \in D(X, \mathscr{A}, \mu)$ such that $P \vartheta_{\ast}(x) = \vartheta_{\ast}(x)$.
\end{corollary}

\begin{definition}\label{DFN6}
If $\vartheta(x) \in D(X, \mathscr{A}, \mu)$, then the entropy of $\vartheta(x)$ is defined by
\begin{align}
 \mathcal{H}(\vartheta(x)) = - \int_{X} \vartheta(x) \ln \vartheta(x) \mu(dx).  \label{EQ18}
\end{align}
\end{definition}

\begin{remark}
The following integral inequality (which is useful for verifying the extreme properties of $\mathcal{H}(\vartheta(x))$) holds for any $\vartheta(x), \xi(x) \in D(X, \mathscr{A}, \mu)$
\begin{align}
 -\int_{X} \vartheta(x) \ln \vartheta(x) \mu(dx) \le -\int_{X} \vartheta(x) \ln \xi(x) \mu(dx).  \label{EQ19}
\end{align}
Note that, in general, we have the following Gibbs inequality
\begin{align*}
 \vartheta(x) - \vartheta(x) \ln \vartheta(x)  \le \xi(x) - \vartheta(x) \ln \xi(x),
\end{align*}
for any two nonnegative measurable functions $\vartheta(x), \xi(x) \in L^1(X, \mathscr{A}, \mu)$.
\end{remark}

\begin{proposition} \label{PR1}
Let $(X, \mathscr{A}, \mu)$ be a finite measure space. Consider all (nonnegative) possible density functions $\vartheta(x)$ defined on $X$. Then, for such a family of density functions, the maximum entropy occurs for a constant density function
\begin{align}
\vartheta_0(x) = 1/\mu(X),  \label{EQ20}
\end{align}
and for any other density function $\vartheta(x)$, the entropy is strictly less than $\mathcal{H}(\vartheta_0(x))$, i.e.,
\begin{align*}
 -\int_{X} \vartheta(x) \ln \vartheta(x) \mu(dx) \le -\ln\left(\frac{1}{\mu(X)}\right).
\end{align*}
\end{proposition}
\begin{proof}
Take any $\vartheta(x) \in D(X, \mathscr{A}, \mu)$. Then, the entropy of $\vartheta(x)$ is given by
\begin{align*}
 \mathcal{H}(\vartheta(x)) = -\int_{X} \vartheta(x) \ln \vartheta(x) \mu(dx).
\end{align*}
Using Equation~\eqref{EQ19}, we have the following
\begin{align*}
 \mathcal{H}(\vartheta(x)) & \le -\int_{X} \vartheta(x) \ln \vartheta_0(x) \mu(dx), \\
                                &= -\ln \left(\frac{1}{\mu(X)}\right) \int_{X} \vartheta(x) \mu(dx),
\end{align*}
and the equality is satisfied only, when $\vartheta(x)=\vartheta_0(x)$. Note that the entropy for $\vartheta_0(x)$ is also given by 
\begin{align*}
 \mathcal{H}(\vartheta_0(x)) &= - \int_{X} \frac{1}{\mu(X)} \ln \left(\frac{1}{\mu(X)}\right) \mu(dx),\\
                                    &= - \ln \left(\frac{1}{\mu(X)}\right). 
\end{align*}
Hence, $\mathcal{H}(\vartheta(x)) \le \mathcal{H}(\vartheta_0(x))$ for all $\vartheta(x) \in D(X, \mathscr{A}, \mu)$. \qed
\end{proof}

\begin{definition}\label{DFN7}
Let $\vartheta(x), \xi(x) \in L^1(X, \mathscr{A}, \mu)$ be two nonnegative measurable functions such that $\operatorname{supp}\xi(x) \subset \operatorname{supp}\vartheta(x)$. Then, the relative entropy of $\xi(x)$ with respect to $\vartheta(x)$ is defined by
\footnote{
{\bf Lemma} (\cite[Voigt (1981)]{Voi81}) Suppose that $P$ is a Markov operator, then
\begin{align*}
 \mathcal{H}_{\rm r}(P^n\xi(x)\, \vert\, P^n \vartheta(x)) \ge \mathcal{H}_{\rm r}(\xi(x)\, \vert\, \vartheta(x)), ~~ \forall \vartheta(x) \in D(X, \mathscr{A}, \mu),
\end{align*}
for any nonnegative measurable function $\xi(x)$.
\begin{remark}
Notice that any linear operator $P \colon L^1(X, \mathscr{A}, \mu) \rightarrow L^1(X, \mathscr{A}, \mu)$ satisfying
\begin{enumerate} [(i)]
\item $P \vartheta(x) \ge 0$ and
\item $\Vert P \vartheta(x)\Vert_{L^1(X, \mathscr{A}, \mu)} = \Vert \vartheta(x)\Vert_{L^1(X, \mathscr{A}, \mu)}$
\end{enumerate}
for any nonnegative measurable function $\vartheta(x) \in L^1(X, \mathscr{A}, \mu)$ is called a Markov operator.
\end{remark}}

\begin{align}
 \mathcal{H}_{\rm r}(\xi(x)\,\vert\,\vartheta(x)) & = \int_{X} \xi(x) \ln \left (\frac{\xi(x)}{\vartheta(x)}\right) \mu(dx),\notag \\
                                                       &= \int_{X} \bigl(\xi(x) \ln \xi(x) - \xi(x) \ln \vartheta(x) \bigr) \mu(dx).  \label{EQ21}
\end{align}
\end{definition}

\begin{remark}
Note that the relative entropy $\mathcal{H}_{\rm r}(\xi(x)\,\vert\,\vartheta(x))$, which measures the deviation of $\xi(x)$ from the density function $\vartheta(x)$, has the following properties.
\begin{enumerate} [(i)]
\item If $\xi(x), \vartheta(x) \in D(X, \mathscr{A}, \mu)$, then $\mathcal{H}_{\rm r}(\xi(x)\,\vert\,\vartheta(x)) \ge 0$ and $\mathcal{H}_{\rm r}(\xi(x)\,\vert\,\vartheta(x)) = 0$ if and only if $\xi(x) = \vartheta(x)$.
\item If $\vartheta(x)$ is constant density and $\vartheta(x) = 1$, then $\mathcal{H}_{\rm r}(\xi(x)\,\vert\,1) = \mathcal{H}(\xi(x))$. Thus, the relative entropy is a generalization of entropy.
\end{enumerate}
\end{remark}

\begin{remark}
For any $\vartheta(x) \in L^1(X, \mathscr{A}, \mu)$, the support of $\vartheta(x)$ is defined by $\operatorname{supp} \vartheta(x) = \bigl\{x \in X \,\vert \, \vartheta(x) \neq 0 \bigr\}$.
\end{remark}

\section{A family of mappings for multi-channel systems} \label{S3}
Consider the following continuous-time multi-channel system  
\begin{align} 
 \dot{x}(t) &= A(t) x(t) + \sum\nolimits_{j \in \mathcal{N}} B_j(t) u_j(t), ~~ x(t_0)=x_0, ~~ t \in [t_0, +\infty ),  \label{EQ22}
\end{align}
where $A(\cdot) \in \mathbb{R}^{d \times d}$, $B_j(\cdot)  \in \mathbb{R}^{d \times r_j}$, $x(t) \in X$ is the state of the system, $u_j(t) \in U_j$ is the control input to the $j$th\,-\,channel and $\mathcal{N} \triangleq \{1, 2, \ldots, N\}$ represents the set of control input channels (or the set of feedback operators) in the system. 

Moreover, we consider the following class of admissible control strategies that will be useful in Section~\ref{S4} (i.e., in a game-theoretic formalism)
\begin{align}
  U_{\mathcal{L}} \subseteq \biggm\{u(t) \in \prod\nolimits_{j \in \mathcal{N}} \underbrace{L^{2}(\mathbb{R}_{+},\mathbb{R}^{r_j}) \cap L^{\infty}(\mathbb{R}_{+},\mathbb{R}^{r_j})}_{\triangleq U_j}\biggm\},  \label{EQ23}
\end{align}
where $u(t)$ is given by $u(t)=\bigl(u_1(t), \, u_{2}(t), \, \ldots, \, u_N(t)\bigr)$.

In what follows, suppose there exists a set of feedback operators $\bigl(\mathcal{L}_1^{\ast},\, \mathcal{L}_2^{\ast},\, \ldots,\, \mathcal{L}_N^{\ast}\bigr)$ from a class of linear operators \,$\mathscr{L} \colon X \rightarrow  U_{\mathcal{L}}$ (i.e., $(\mathcal{L}_jx)(t) \in U_j$ for $j \in \mathcal{N}$) with strategies $\bigl(\mathcal{L}_j^{\ast}x\bigr)(t) \in U_j$ for $t \ge t_0$ and for $j \in \mathcal{N}$. Further, let $\phi_j\bigl(t; t_0, x_0, \bigl(\widehat{u_j(t), u^{\ast}_{\neg j}(t)}\bigr)\bigr) \in X$ be the unique solution of the $j$th\,-\,subsystem
\begin{align}
 \dot{x}^j(t) &= \Bigl(A(t)+ \sum\nolimits_{i \in \mathcal{N}_{\neg j}} B_i(t) \mathcal{L}_i^{\ast}(t) \Bigr) x^j(t) + B_j(t) u_j(t),  \label{EQ24}
\end{align}
with an initial condition $x_0 \in X$ and control inputs given by
\begin{align}
\bigl(\widehat{u_j(t), u^{\ast}_{\neg j}(t)}\bigr)\triangleq\bigl(u^{\ast}_1(t),\, u^{\ast}_{2}(t),\,\ldots,\,u^{\ast}_{j-1}(t),\, u_{j}(t),\, u^{\ast}_{j+1}(t),\,\ldots,\, u^{\ast}_N(t)\bigr)\in U_{\mathcal{L}},   \label{EQ25}
\end{align}
where $u^{\ast}_i(t) = \mathcal{L}_i^{\ast}(t) x^j(t)$ for $i \in \mathcal{N}_{\neg j} \triangleq \mathcal{N} \backslash \{j\}$ and $j \in \mathcal{N}$. 

Furthermore, we may require that the control input for the $j$th\,-\,channel to be $u_j(t)=\bigl(\mathcal{L}_jx^j\bigr)(t) \in U_j$ and with this set of linear feedback operators
\begin{align*}
\underbrace{\bigl(\mathcal{L}_1^{\ast},\,\ldots,\, \mathcal{L}_{j-1}^{\ast},\,\mathcal{L}_{j},\,\mathcal{L}_{j+1}^{\ast},\, \ldots,\,\mathcal{L}_N^{\ast}\bigr)}_{ \triangleq \bigl(\mathcal{L}_j, \mathcal{L}_{\neg j}^{\ast}\bigr)} \in \mathscr{L}.
\end{align*}
Then, the unique solution $\phi_j\bigl(t; t_0, x_0, \bigl(\widehat{u_j(t), u^{\ast}_{\neg j}(t)}\bigr)\bigr)$ will take the form
\begin{align}
 \phi_j\bigl(t; t_0, x_0, \bigl(\widehat{u_j(t), u^{\ast}_{\neg j}(t)}\bigr)\bigr) = \underbrace{\Phi^{\mathcal{L}_{\neg j}^{\ast}}(t, t_0)\, \Phi^{\mathcal{L}_j}(t, t_0)}_{\triangleq \Phi_t^{(\mathcal{L}_j, \mathcal{L}_{\neg j}^{\ast})}} x_0, ~~ \forall t \in [t_0,\,+\infty),  \label{EQ26}
\end{align}
where
\begin{align}
 \frac{\partial\Phi^{\mathcal{L}_{\neg j}^{\ast}}(t, \tau)}{\partial t} &= \Bigl(A(t) + \sum\nolimits_{i \in \mathcal{N}_{\neg j}} B_i(t) \mathcal{L}_i^{\ast}(t) \Bigr)\Phi^{\mathcal{L}_{\neg j}^{\ast}}(t, \tau),  \label{EQ27} \\
 \frac{\partial\Phi^{\mathcal{L}_j}(t, \tau)}{\partial t} &= B_j^{\ast}(t)\Phi^{\mathcal{L}_j}(t, \tau),   \label{EQ28}
\end{align}
with both $\Phi^{\mathcal{L}_{\neg j}^{\ast}}(\tau, \tau)$ and $\Phi^{\mathcal{L}_j}(\tau, \tau)$ are identity matrices; and $B^{\ast}(t)$ is given by
\begin{align}
 B_j^{\ast}(t) = \Bigl(\Phi^{\mathcal{L}_{\neg j}^{\ast}}(t, \tau)\Bigr)^{-1} B_j(t)\mathcal{L}_j(t) \Phi^{\mathcal{L}_{\neg j}^{\ast}}(t, \tau),  \label{EQ29}
 \end{align}
for each $j \in \mathcal{N}$.\footnote{Note that, with $t_0 = \tau$, if we take the partial derivative of $\Phi_t^{(\mathcal{L}_j,\,\mathcal{L}_{\neg j}^{\ast})}$ with respect to $t$ and make use of Equations~\eqref{EQ6} and \eqref{EQ7} together with Equation~\eqref{EQ8}, then we have
\begin{align*}
\frac{\partial}{\partial t}\Bigl(\Phi_t^{(\mathcal{L}_j,\,\mathcal{L}_{\neg j}^{\ast})}\Bigr) &= \frac{\partial}{\partial t}\Bigl(\Phi^{\mathcal{L}_{\neg j}^{\ast}}(t, \tau)\Bigr)\,\Phi^{\mathcal{L}_j}(t, \tau) + \Phi^{\mathcal{L}_{\neg j}^{\ast}}(t, \tau)\, \frac{\partial}{\partial t}\Bigl(\Phi^{\mathcal{L}_j}(t, \tau))\Bigr),  \\
&= \Bigl(A(t) + \sum\nolimits_{i \in \mathcal{N}_{\neg j}} B_i(t) \mathcal{L}_i^{\ast}(t) \Bigr)\Phi^{\mathcal{L}_{\neg j}^{\ast}}(t, \tau)\, \Phi^{\mathcal{L}_j}(t, \tau) \\
& \quad\quad\quad \quad + \Phi^{\mathcal{L}_{\neg j}^{\ast}}(t, \tau) \Bigl(\Phi^{\mathcal{L}_{\neg j}^{\ast}}(t, \tau)\Bigr)^{-1} B_j(t)\mathcal{L}_j(t)\,\Phi^{\mathcal{L}_{\neg j}^{\ast}}(t, \tau) \,\Phi^{\mathcal{L}_j}(t, \tau)),  \\
&= \left(\Bigl(A(t) + \sum\nolimits_{i \in \mathcal{N}_{\neg j}} B_i(t) \mathcal{L}_i^{\ast}(t) \Bigr) +  B_j(t)\mathcal{L}_j(t) \right) \,\Phi_t^{(\mathcal{L}_j,\,\mathcal{L}_{\neg j}^{\ast})}.
\end{align*}
Moreover, we note that $\Phi^{\mathcal{L}_{\neg j}^{\ast}}(t, \tau)$ satisfies the following
\begin{align*}
 \Phi^{\mathcal{L}_{\neg j}^{\ast}}(t_2, t_1)\, \Phi^{\mathcal{L}_{\neg j}^{\ast}}(t_1, \tau) = \Phi^{\mathcal{L}_{\neg j}^{\ast}}(t_2, \tau), \quad \forall t_1, t_2 \in [\tau,\,+\infty),
\end{align*}
(e.g., see \cite{Che62} for such a decomposition that arises in differential equations).}

In the following, we assume that $X$ is a topological Hausdorff space and $\mathscr{A}$ is a $\sigma$\,-\,algebra of Borel set, i.e., the smallest $\sigma$\,-\,algebra which contains all open, and thus closed, subsets of $X$. With this, for any $t \ge 0$ (assuming that $t_0=0)$ and $(\mathcal{L}_j, \mathcal{L}_{\neg j}^{\ast})\in \mathscr{L}$, we can consider a family of continuous mappings (or transformations) $\Bigl\{\Phi_t^{(\mathcal{L}_j, \mathcal{L}_{\neg j}^{\ast})} \Bigr\}_{t \ge 0}$
on $X$ satisfying
\begin{align}
 (\mathbb{R}_{+} \times X) \ni (t,\,x) \mapsto \Phi_t^{(\mathcal{L}_j, \mathcal{L}_{\neg j}^{\ast})} x \in X.  \label{EQ30} 
\end{align}
Note that, for each fixed $t \ge 0$, the transformation $\Phi_t^{(\mathcal{L}_j, \mathcal{L}_{\neg j}^{\ast})}$ is measurable, i.e., we have $\Bigl(\Phi_t^{(\mathcal{L}_j, \mathcal{L}_{\neg j}^{\ast})}\Bigr)^{-1}(A) \in \mathscr{A}$ for all $A \in \mathcal{A}$, where $\Bigl(\Phi_t^{(\mathcal{L}_j, \mathcal{L}_{\neg j}^{\ast})}\Bigr)^{-1}(A)$ denotes the set of all $x$ such that $\Phi_t^{(\mathcal{L}_j, \mathcal{L}_{\neg j}^{\ast})} x \in A$ (e.g., see \cite{LasPia77} for invariant measures on topological spaces; see also \cite{BauSig75} for topological properties of measure spaces).

Then, we can introduce the following definitions.
\begin{definition}\label{DFN8}
A measure $\mu$ is called invariant under the family of measurable transformations $\Bigl\{\Phi_t^{(\mathcal{L}_j, \mathcal{L}_{\neg j}^{\ast})} \Bigr\}_{t \ge 0}$ if
\begin{align}
 \mu\Bigl(\Phi_t^{(\mathcal{L}_j, \mathcal{L}_{\neg j}^{\ast})} \bigr)^{-1}(A) = \mu\bigl(A\bigr), ~~ \forall A \in \mathcal{A}.  \label{EQ31}
\end{align}
\end{definition}

Next, we assume that the family of transformations $\Bigl\{\Phi_t^{(\mathcal{L}_j, \mathcal{L}_{\neg j}^{\ast})} \Bigr\}_{t \ge 0}$ are nonsingular and, for each fixed $t \ge 0$, the unique Frobenius-Perron operator $\mathit{P}_t^{(\mathcal{L}_j, \mathcal{L}_{\neg j}^{\ast})} \colon L^1(X, \mathscr{A}, \mu)\rightarrow L^1(X, \mathscr{A}, \mu)$ is then defined by 
\begin{align}
 \int_{A} \mathit{P}_t^{(\mathcal{L}_j, \mathcal{L}_{\neg j}^{\ast})} \vartheta(x) \mu(dx) =  \int_{\Bigl(\Phi_t^{(\mathcal{L}_j, \mathcal{L}_{\neg j}^{\ast})} \bigr)^{-1}(A) } \vartheta(x) \mu(dx), ~~ \forall A \in \mathscr{A}.  \label{EQ32}
 \end{align}

\begin{definition}\label{DFN9}
Let $(X, \mathscr{A}, \mu)$ be a measure space, then the family of operators $\Bigl\{\mathit{P}_t^{(\mathcal{L}_j, \mathcal{L}_{\neg j}^{\ast})}\Bigr\}_{t \ge 0}$, $\forall (\mathcal{L}_j, \mathcal{L}_{\neg j}^{\ast}) \in \mathscr{L}$, $\forall t \ge 0$ and $\forall j \in \mathcal{N}$, satisfies the following properties.
\begin{enumerate}[(P1)]
\item For all $\vartheta_1(x), \vartheta_2(x) \in L^1(X, \mathscr{A}, \mu)$ and $\lambda_1, \lambda_2 \in \mathbb{R}$ 
\begin{align*}
 \mathit{P}_t^{(\mathcal{L}_j, \mathcal{L}_{\neg j}^{\ast})} \Bigl\{\lambda_1 \vartheta_1(x) + \lambda_2 \vartheta_2(x) \Bigr\}= \lambda_1 \mathit{P}_t^{(\mathcal{L}_j, \mathcal{L}_{\neg j}^{\ast})} \vartheta_1(x) + \lambda_2  \mathit{P}_t^{(\mathcal{L}_j, \mathcal{L}_{\neg j}^{\ast})} \vartheta_2(x). 
\end{align*}
\item $\mathit{P}_t^{(\mathcal{L}_j, \mathcal{L}_{\neg j}^{\ast})} \vartheta(x) \ge 0$ if $\vartheta(x) \ge 0$.
\item For all $\vartheta(x) \in L^1(X, \mathscr{A}, \mu)$,
\begin{align*}
 \int_{X} \mathit{P}_t^{(\mathcal{L}_j^{\ast}, \mathcal{L}_{\neg j}^{\ast})} \vartheta(x) \mu(dx) = \int_{X} \vartheta(x)\mu(dx),
\end{align*}
\end{enumerate}
is called a semigroup; and it is uniformly continuous, if 
\begin{align}
 \lim_{t \rightarrow  t_0} \Bigl \Vert \mathit{P}_t^{(\mathcal{L}_j^{\ast}, \mathcal{L}_{\neg j}^{\ast})} \vartheta(x) -  \mathit{P}_{t_0}^{(\mathcal{L}_j^{\ast}, \mathcal{L}_{\neg j}^{\ast})} \vartheta(x) \Bigr\Vert_{L^1(X, \mathscr{A}, \mu)} = 0, ~~ \forall t_0 \ge 0,  \label{EQ33}
\end{align}
for each $\vartheta(x) \in L^1(X, \mathscr{A}, \mu)$ with $(\mathcal{L}_j^{\ast}, \mathcal{L}_{\neg j}^{\ast}) \in \mathscr{L}$.
\end{definition}

\section{Main Results} \label{S4}
\subsection{Game-theoretic formalism} \label{S4(1)}
In the following, we specify a game in a feedback strategic form -- where, in the course of the game, each feedback operator generates automatically a feedback control in response to the action of other feedback operators via the system state $x(t)$ for $t \in [t_0,\,+\infty)$. For example, the $j$th\,-\,feedback operator can generate a feedback control $u_j(t)=\bigl(\mathcal{L}_jx^j\bigr)(t)$ in response to the actions of other feedback operators $u^{\ast}_i(t)= \bigl(\mathcal{L}_i^{\ast} x^j\bigr)(t)$ for $i \in \mathcal{N}_{\neg j}$, where $\bigl(\widehat{u_j(t), u^{\ast}_{\neg j}(t)}\bigr) \in  U_{\mathcal{L}}$, and, similarly, any number of feedback operators can decide on to play feedback strategies simultaneously. Hence, for such a game to have a set of stable (game-theoretic) equilibrium feedback operators (which is also robust to small perturbations in the system or strategies played by others), then each feedback operator is required to respond (in some sense of best-response correspondence) to the others strategies.

To this end, it will be useful to consider the following criterion functions
\begin{align}
\mathscr{L} \ni (\mathcal{L}_j, \mathcal{L}_{\neg j}^{\ast}) \mapsto \mathcal{H}_{\rm r}\Bigr(\mathit{P}_t^{(\mathcal{L}_j, \mathcal{L}_{\neg j}^{\ast})}\vartheta(x)\,\bigl\lvert\,\vartheta(x) \Bigl) \in \mathbb{R}_{-} \cup \{-\infty\}, ~~ \forall t\ge 0, ~~ \forall j \in \mathcal{N}, \label{EQ34} 
\end{align}
over the class of admissible control functions $U_{\mathcal{L}}$ (or the set of linear feedback operators $\mathscr{L}$) and for any $\vartheta(x) \in D(X, \mathscr{A}, \mu)$.

Note that, under the game-theoretic framework, if there exists a set of equilibrium feedback operators $\bigl(\mathcal{L}_1^{\ast},\, \mathcal{L}_2^{\ast},\, \ldots,\, \mathcal{L}_N^{\ast}\bigr)$ (from the class of linear feedback operators $\mathscr{L}$). Then, this set of equilibrium feedback operators decreases the relative entropy between any two density functions from $D(X, \mathscr{A}, \mu)$ for all $t \ge 0$. On the other hand, if there exists a unique stationary density function (i.e., a common fixed-point) $\vartheta_{\ast}(x) \in D(X, \mathscr{A}, \mu)$ for the family of Frobenius-Perron operators $\mathit{P}_t^{(\mathcal{L}_j^{\ast}, \mathcal{L}_{\neg j}^{\ast})}$ for each fixed $t \ge 0$. Then, the composition of the multi-channel system with this set of equilibrium feedback operators, when described by density functions, will evolve towards the unique equilibrium state, i.e., the entropy of the whole system will be maximized (see Lanford \cite[pp\,1--113]{Lan73} for an exposition of equilibrium states and entropy in statistical mechanics).

Therefore, more formally, we have the following definition for the set of equilibrium feedback operators $(\mathcal{L}_1^{\ast},\, \mathcal{L}_2^{\ast}, \, \ldots, \, \mathcal{L}_N^{\ast}) \in \mathscr{L}$.

\begin{definition}\label{DFN10}
We shall say that a set of linear system operators $(\mathcal{L}_1^{\ast},\, \mathcal{L}_2^{\ast}, \, \ldots, \, \mathcal{L}_N^{\ast}) \in \mathscr{L}$ is a set of (game-theoretic) equilibrium feedback operators, if they produce control responses given by
\begin{align}
 u_j^{\ast}(t) = \bigl(\mathcal{L}_j^{\ast}x\bigr)(t) \in U_j, ~~ \forall j \in \mathcal{N}, \label{EQ35}
\end{align}
 for $t \in [0, \infty]$ and satisfy further the following conditions
\begin{align}
\left.\begin{array}{r}
\mathcal{H}_{\rm r}\Bigl(\mathit{P}_t^{(\mathcal{L}_j, \mathcal{L}_{\neg j}^{\ast})}\vartheta(x)\,\bigl\lvert\,\vartheta(x) \Bigr) \ge \mathcal{H}_{\rm r}\Bigl(\mathit{P}_t^{(\mathcal{L}_j^{\ast}, \mathcal{L}_{\neg j}^{\ast})}\vartheta(x)\,\bigl\lvert\,\vartheta(x) \Bigr), \quad \forall t \ge 0, ~~ \\
  ~~~~ \forall (\mathcal{L}_j, \mathcal{L}_{\neg j}^{\ast}) \in \mathscr{L}, ~ \forall j \in \mathcal{N}, ~~ \\  \\
 D(X, \mathscr{A}, \mu)\ni \mathit{P}_t^{(\mathcal{L}_j^{\ast}, \mathcal{L}_{\neg j}^{\ast})}\vartheta(x) \rightarrow \vartheta_{\ast}(x) \in D(X, \mathscr{A}, \mu) \quad  \text{as} \quad t \rightarrow \infty, ~~ \\ \\
 \mathcal{H}\Bigl(\mathit{P}_t^{(\mathcal{L}_j, \mathcal{L}_{\neg j}^{\ast})}\vartheta(x)\Bigr) \le \mathcal{H}\Bigl(\mathit{P}_t^{(\mathcal{L}_j^{\ast}, \mathcal{L}_{\neg j}^{\ast})}\vartheta_{\ast}(x)\Bigr), \quad  \forall t \ge 0, \quad \forall j \in \mathcal{N}, ~~
 \end{array} \right \} \label{EQ36}
\end{align}
for each $\vartheta(x) \in D(X, \mathscr{A}, \mu)$.
\end{definition}

\begin{remark}
We remark that the relative entropy $\mathcal{H}_{\rm r}(\cdot \lvert \cdot)$ in Equation~\eqref{EQ36} is determined with respect to $\mathcal{L}_j$ for each $j \in \mathcal{N}$; while the others $L_{\neg j}^{\ast}$ remain fixed.
\end{remark}

Then, we formally state the main objective of this paper.
\begin{problem}
Provide a sufficient condition for the existence of a set of equilibrium feedback operators in the multi-channel system (that interacts strategically in a game-theoretic framework) such that when the composition of the multi-channel system with this set of equilibrium feedback operators, described by density functions, will evolve towards an equilibrium state in such a way that the entropy of the whole system is maximized.
\end{problem}

\subsection{Existence of a set of (game-theoretic) equilibrium feedback operators}\label{S4(2)}

In the following, we provide a sufficient condition for the existence of a unique equilibrium state that is associated with a stationary density function (i.e., a common fixed-point) for the family of Frobenius-Perron operators $\Bigl\{\mathit{P}_t^{(\mathcal{L}_j^{\ast}, \mathcal{L}_{\neg j}^{\ast})}\Bigr\}_{t \ge 0}$.
\begin{proposition}\label{PR2}
Let $B(X, \mathscr{A}, \mu)$ be an open ball in $D(X, \mathscr{A}, \mu)$ of center $\vartheta_0(x) \in D(X, \mathscr{A}, \mu)$ and radius $\beta$, i.e,,
\begin{align}
 B(X, \mathscr{A}, \mu) = \Bigl \{ \vartheta(x) \in D(X, \mathscr{A}, \mu) \, \Bigl \vert \, \bigl\lVert \vartheta(x) - \vartheta_0(x) \bigr\rVert_{L^1(X, \mathscr{A}, \mu)} \le \beta \Bigr\}. \label{EQ37} 
\end{align}
Suppose that there exists a set of feedback operators $\bigr(\mathcal{L}_1^{\ast}, \mathcal{L}_2^{\ast}, \ldots, \mathcal{L}_N^{\ast} \bigl)\,\in\negthinspace\mathscr{L}$ such that the family of Frobenius-Perron operators $\Bigr\{\mathit{P}_t^{(\mathcal{L}_j, \mathcal{L}_{\neg j}^{\ast})} \Bigl\}_{t \ge 0}$ with respect to $\Phi_t^{(\mathcal{L}_j, \mathcal{L}_{\neg j}^{\ast})}$ satisfies
\begin{align}
\sup_{(\mathcal{L}_j, \mathcal{L}_{\neg j}^{\ast}) \in \mathscr{L}} \Bigl \Vert \mathit{P}_t^{(\mathcal{L}_j, \mathcal{L}_{\neg j}^{\ast})} \vartheta_2(x) &- \mathit{P}_t^{(\mathcal{L}_j, \mathcal{L}_{\neg j}^{\ast})} \vartheta_1(x) \Bigr\Vert_{L^1(X, \mathscr{A}, \mu)} \notag \\
&\le \kappa \bigl \Vert \vartheta_2(x) - \vartheta_1(x) \bigr\Vert_{L^1(X, \mathscr{A}, \mu)}, ~ \forall t \ge 0, ~ \forall j \in \mathcal{N}, \label{EQ38} 
\end{align}
for any two $\vartheta_1(x), \vartheta_2(x) \in B(X, \mathscr{A}, \mu)$, where $\kappa$ is a positive constant which is less than one.  

Then, if 
\begin{align}
\sup_{(\mathcal{L}_j, \mathcal{L}_{\neg j}^{\ast}) \in \mathscr{L}} \Bigl \Vert \mathit{P}_t^{(\mathcal{L}_j, \mathcal{L}_{\neg j}^{\ast})} \vartheta_0(x) - \vartheta_0(x) \Bigr\Vert_{L^1(X, \mathscr{A}, \mu)} \le \beta\bigl(1 - \kappa \bigr), ~~  \forall t \ge 0, ~~  \forall j \in \mathcal{N}, \label{EQ39} 
\end{align}
there is at least one stationary density function (i.e., a common fixed-point) ~ $\vartheta_{\ast}(x) \in B(X, \mathscr{A}, \mu)$ such that
\begin{align}
 \mathit{P}_t^{(\mathcal{L}_j^{\ast}, \mathcal{L}_{\neg j}^{\ast})} \vartheta_{\ast}(x) = \vartheta_{\ast}(x), ~~ \forall t \ge 0, ~ \forall j \in \mathcal{N}. \label{EQ40}
 \end{align}
Furthermore, there exists a unique equilibrium state, which corresponds with $\vartheta_{\ast}(x)$, if the measure $\mu_{\ast}$
\begin{align}
  \mu_{\ast}(A) = \int_{A} \vartheta_{\ast}(x) \mu(dx), ~~ \forall A \in \mathscr{A}, \label{EQ41} 
\end{align}
is invariant with respect to $\Phi_t^{(\mathcal{L}_j^{\ast}, \mathcal{L}_{\neg j}^{\ast})}$ for each fixed $t \ge 0$.\footnote{Note that the supremum in Equation~\eqref{EQ38} (and also in Equation~\eqref{EQ39}) is computed with respect to $\mathcal{L}_j$ with $(\mathcal{L}_j, \mathcal{L}_{\neg j}^{\ast}) \in \mathscr{L}$ for each $j \in \mathcal{N}$, while $\mathcal{L}_{\neg j}^{\ast}$ remains fixed.}
\end{proposition}

\begin{proof}
Observe that $\mathit{P}_t^{(\mathcal{L}_j, \mathcal{L}_{\neg j}^{\ast})}$ is continuous for each $t \ge 0$ and for any $\vartheta(x) \in D(X, \mathscr{A}, \mu)$ (cf. Equation~\eqref{EQ33}). For $\vartheta_{\ast}(x) \in B(X, \mathscr{A}, \mu)$, we will show that there exists a convergent sequence of functions $\bigl\{ \vartheta_n(x)\bigr\}$ such that
\begin{align*}
 B(X, \mathscr{A}, \mu) \ni  \vartheta_n(x) = \mathit{P}_t^{(\mathcal{L}_j, \mathcal{L}_{\neg j}^{\ast})}\vartheta_{n-1}(x) - \vartheta_0(x),  ~ \forall n \ge 0, ~ \forall j \in \mathcal{N}, ~ \forall (\mathcal{L}_j, \mathcal{L}_{\neg j}^{\ast}) \in \mathscr{L},
\end{align*}
We see that if $\vartheta_p(x)$ is defined in $B(X, \mathscr{A}, \mu)$ for $1 \le p \le n$, i.e., $\mathit{P}_t^{(\mathcal{L}_j, \mathcal{L}_{\neg j}^{\ast})}\vartheta_{p-1}(x) -\vartheta_0(x) \in B(X, \mathscr{A}, \mu)$, $\forall p \in [1, n]$, then we have followings
\begin{align*}
  \vartheta_p(x) - \vartheta_{p-1}(x) = \mathit{P}_t^{(\mathcal{L}_j, \mathcal{L}_{\neg j}^{\ast})}\vartheta_{p-1}(x) - \mathit{P}_t^{(\mathcal{L}_j, \mathcal{L}_{\neg j}^{\ast})}\vartheta_{p-2}(x).
\end{align*}
and 
\begin{align*}
 &\bigl \Vert \vartheta_p(x) - \vartheta_{p-1}(x) \bigr\Vert_{L^1(X, \mathscr{A}, \mu)} \notag \\
  &\quad\quad\quad \le \kappa \sup_{(\mathcal{L}_j, \mathcal{L}_{\neg j}^{\ast}) \in \mathscr{L}}\bigl\Vert \mathit{P}_t^{(\mathcal{L}_j, \mathcal{L}_{\neg j}^{\ast})}\vartheta_{p-1}(x) - \mathit{P}_t^{(\mathcal{L}_j, \mathcal{L}_{\neg j}^{\ast})}\vartheta_{p-2}(x)\bigr \Vert_{L^1(X, \mathscr{A}, \mu)}, \notag \\
  &\quad\quad\quad\quad\quad\quad\quad\quad ~~ ~ \forall t \ge 0, ~ \forall j \in \mathcal{N}.
\end{align*}
With $(\mathcal{L}_j^{\ast}, \mathcal{L}_{\neg j}^{\ast}) \in \mathscr{L}$, we conclude that
\begin{align*}
 \bigl \Vert \vartheta_p(x) - \vartheta_{p-1}(x) \bigr\Vert_{L^1(X, \mathscr{A}, \mu)} \le \kappa^{p-1}\bigl\Vert\vartheta_1(x)\bigr \Vert_{L^1(X, \mathscr{A}, \mu)},
\end{align*}
which further gives us
\begin{align*}
 \bigl \Vert \vartheta_p(x)\bigr\Vert_{L^1(X, \mathscr{A}, \mu)} \le (1+ \kappa + \kappa^2+ \cdots + \kappa^{p-1})\bigl\Vert\vartheta_1(x)\bigr \Vert_{L^1(X, \mathscr{A}, \mu)},
\end{align*}
and
\begin{align*}
 \bigl \Vert \vartheta_p(x)\bigr\Vert_{L^1(X, \mathscr{A}, \mu)} \le \frac{1}{1- \kappa} \bigl\Vert\vartheta_1(x)\bigr \Vert_{L^1(X, \mathscr{A}, \mu)} < \beta.
\end{align*}
Hence, this agrees with our claim, i.e.,
\begin{align*}
\sup_{(\mathcal{L}_j, \mathcal{L}_{\neg j}^{\ast}) \in \mathscr{L}}\bigl\Vert \mathit{P}_t^{(\mathcal{L}_j, \mathcal{L}_{\neg j}^{\ast})}\vartheta_0(x) - \vartheta_0(x)\bigr \Vert_{L^1(X, \mathscr{A}, \mu)} < \beta(1- \kappa), ~~ \forall t \ge 0, ~ \forall j \in \mathcal{N}.
\end{align*}
Note that, for any $n \ge 0$, we have
\begin{align*}
 \bigl \Vert \vartheta_n(x) - \vartheta_{n-1}(x) \bigr\Vert_{L^1(X, \mathscr{A}, \mu)} \le \kappa^{n-1}\bigl\Vert\vartheta_1(x)\bigr \Vert_{L^1(X, \mathscr{A}, \mu)},
\end{align*}
which is strongly convergent (i.e., $\lim_{n \rightarrow \infty}  \bigl \Vert \vartheta_n(x) - \vartheta_{n-1}(x) \bigr\Vert_{L^1(X, \mathscr{A}, \mu)} = 0$). Then, by passing to a limit, we conclude that there exists a common fixed-point (or a stationary density function) $\vartheta_{\ast}(x) \in B(X, \mathscr{A}, \mu)$ for the family of Frobenius-Perron operators $\Bigl\{\mathit{P}_t^{(\mathcal{L}_j^{\ast}, \mathcal{L}_{\neg j}^{\ast})} \Bigr\}_{t \ge 0}$ that satisfies
\begin{align*}
 \mathit{P}_t^{(\mathcal{L}_j^{\ast}, \mathcal{L}_{\neg j}^{\ast})} \vartheta_{\ast}(x) = \vartheta_{\ast}(x), ~~ \forall t \ge 0,
\end{align*} 
which also corresponds to the unique equilibrium state, in the sense of statistical mechanics, for the multi-channel system together with $(\mathcal{L}_j^{\ast}, \mathcal{L}_{\neg j}^{\ast}) \in \mathscr{L}$.\footnote{\label{FT1}Note that the following also holds true
\begin{align*}
 \lim_{t \rightarrow \infty} \mathit{P}_t^{(\mathcal{L}_j^{\ast}, \mathcal{L}_{\neg j}^{\ast})} \vartheta(x) = \vartheta_{\ast}(x),
\end{align*}
for any $\vartheta(x) \in B(X, \mathscr{A}, \mu)$.}

Moreover, from Theorem~\ref{TH1}, we see that the measure $\mu_{\ast}$, i.e.,
\begin{align*}
  \mu_{\ast}(A) = \int_{A} \vartheta_{\ast}(x) \mu(dx), ~~ \forall A \in \mathscr{A},
\end{align*}
 is invariant with respect to $\Phi_t^{(\mathcal{L}_j^{\ast}, \mathcal{L}_{\neg j}^{\ast})}$ for each fixed $t \ge 0$.\qed
\end{proof}

The above proposition (i.e., Proposition~\ref{PR2}) is important because of the three way connection it draws between the existence of a common stationary density function $\vartheta_{\ast}(x) \in D(X, \mathscr{A}, \mu)$ for $\Bigl\{\mathit{P}_t^{(\mathcal{L}_j^{\ast}, \mathcal{L}_{\neg j}^{\ast})}\Bigr\}_{t \ge 0}$ (i.e., the unique equilibrium state), the invariant measure $\mu_{\ast}$ (i.e., the measure preserving property of $\Phi_t^{(\mathcal{L}_j^{\ast}, \mathcal{L}_{\neg j}^{\ast})}$ for all $t \ge 0$) and the set of equilibrium feedback operators $\bigr(\mathcal{L}_1^{\ast}, \mathcal{L}_2^{\ast}, \ldots, \mathcal{L}_N^{\ast} \bigl) \in \mathscr{L}$. Moreover, the corresponding maximum entropy $\mathcal{H}_{max} \bigl(\vartheta_{\ast}(x) \bigr)$ is given by
\begin{align*}
  \mathcal{H}_{\rm max} \bigl(\vartheta_{\ast}(x) \bigr) = -\int_{X} \vartheta_{\ast}(x) \ln \vartheta_{\ast}(x) \mu(dx).
\end{align*}

\begin{remark}
We remark that, in the above proposition, a fixed-point theorem is implicitly used for deriving a sufficient condition for the existence of a common stationary density function for the family of Frobenius-Perron operators (e.g., see Dunford and Schwartz \cite[pp\,456]{DunSch58} or Dieudonn\'{e} \cite[pp\,261]{Dieu60}).
\end{remark}

\subsection{Asymptotic stability of the family of Frobenius-Perron operators} \label{S4(3)}

Here, we provide a connection between the stationary density function $\vartheta_{\ast}(x)  \in D(X, \mathscr{A}, \mu)$ (which corresponds to the equilibrium state) and the asymptotic stability of the family of Frobenius-Perron operators $\Bigl\{\mathit{P}_t^{(\mathcal{L}_j^{\ast}, \mathcal{L}_{\neg j}^{\ast})} \Bigr\}_{t \ge 0}$. Note that, from Proposition~\ref{PR2}, any initial density function $\vartheta(x) \in D(X, \mathscr{A}, \mu)$ under the action of the family of Frobenius-Perron operators $\Bigl\{\mathit{P}_t^{(\mathcal{L}_j, \mathcal{L}_{\neg j}^{\ast})} \Bigr\}_{t \ge 0}$ will only converge to a unique stationary density function $\vartheta_{\ast}(x)  \in D(X, \mathscr{A}, \mu)$, if the relative entropy 
\begin{align}
\sup_{(\mathcal{L}_j, \mathcal{L}_{\neg j}^{\ast}) \in \mathscr{L}} \mathcal{H}_{\rm r}\Bigr(\mathit{P}_t^{(\mathcal{L}_j, \mathcal{L}_{\neg j}^{\ast})}\vartheta(x)\,\bigl\lvert\,\vartheta_{\ast}(x) \Bigl), ~~ \forall j \in \mathcal{N}, \label{EQ42} 
\end{align}
tends zero as $t \rightarrow \infty$, and when the set of feedback operators attains a (game-theoretic) equilibrium.

Then, we have the following corollary that exactly establishes the connection between the relative entropy and the stationary density function (where the latter corresponds to the unique equilibrium state).
\begin{corollary}\label{CR4}
Suppose that the set of equilibrium feedback operators $\bigl(\mathcal{L}_1^{\ast}, \mathcal{L}_2^{\ast}, \ldots, \mathcal{L}_N^{\ast} \bigr) \in \mathscr{L}$ satisfies Proposition~\ref{PR2}. Then, 
\begin{align}
 \lim_{t \rightarrow \infty} \mathcal{H}_{\rm r}\Bigl(\mathit{P}_t^{(\mathcal{L}_j^{\ast}, \mathcal{L}_{\neg j}^{\ast})}\vartheta(x)\,\bigl\lvert\,\vartheta_{\ast}(x) \Bigr) = 0, ~~ \forall j \in \mathcal{N}, \label{EQ43} 
\end{align}
for each $\vartheta(x) \in D(X, \mathscr{A}, \mu)$ such that $\mathcal{H}_{\rm r}\bigr(\vartheta(x)\,\bigl\lvert\,\vartheta_{\ast}(x) \bigl)$ is finite.
\end{corollary}

\begin{proof}
From Proposition~\ref{PR2}, if $\bigl(\mathcal{L}_1^{\ast}, \mathcal{L}_2^{\ast}, \ldots, \mathcal{L}_N^{\ast} \bigr) \in \mathscr{L}$ is a set of equilibrium feedback operators. Then, there is a common fixed-point density function $\vartheta_{\ast}(x) \in D(X, \mathscr{A}, \mu)$ such that
\begin{align*}
 \vartheta_{\ast}(x) \in \bigcap_{t \ge 0} \overline{\mathit{P}_t^{(\mathcal{L}_j, \mathcal{L}_{\neg j}^{\ast})}\vartheta(x)} \neq \varnothing, ~~ \forall (\mathcal{L}_j, \mathcal{L}_{\neg j}^{\ast}) \in \mathscr{L}, ~~ \forall j \in \mathcal{N}, \\
  \forall \vartheta(x) \in D(X, \mathscr{A}, \mu),
\end{align*}
and
\begin{align*}
 \mathit{P}_t^{(\mathcal{L}_j^{\ast}, \mathcal{L}_{\neg j}^{\ast})}\vartheta_{\ast}(x) = \vartheta_{\ast}(x), ~~ \forall t \ge 0,
\end{align*}
with (cf. Equation~\eqref{EQ36} or Footnote~\ref{FT1})
\begin{align*}
 \lim_{t \rightarrow \infty} \mathit{P}_t^{(\mathcal{L}_j^{\ast}, \mathcal{L}_{\neg j}^{\ast})}\vartheta(x) = \vartheta_{\ast}(x), ~~ \forall \vartheta(x) \in D(X, \mathscr{A}, \mu).
\end{align*}
Further, if $\mathcal{H}_{\rm r}\bigr(\vartheta(x)\,\bigl\lvert\,\vartheta_{\ast}(x) \bigl)$ is finite, then we have
\begin{align*}
\mathcal{H}_{\rm r}\Bigl(\mathit{P}_t^{(\mathcal{L}_j^{\ast}, \mathcal{L}_{\neg j}^{\ast})}\vartheta(x)\,\bigl\lvert\,\vartheta_{\ast}(x) \Bigr) \rightarrow 0 \quad  \text{as} \quad t \rightarrow \infty,
\end{align*}
for any $ \vartheta(x) \in D(X, \mathscr{A}, \mu)$.\qed
\end{proof}

\subsection{Resilient behavior of a set of (game-theoretic) equilibrium feedback operators} \label{S4(4)}
Here, we consider the following systems with a small random perturbation term
\begin{align} 
 d Z_{\epsilon}^j(t) = \Bigl(A(t) +  \sum\nolimits_{i \in \mathcal{N}_{\neg j}} B_i(t) \mathcal{L}_i^{\ast}(t) \Bigr)Z_{\epsilon}^j(t) dt  + B_j(t)\mathcal{L}_i(t) Z_{\epsilon}^j(t) dt \notag \\
   \quad + \sqrt{\epsilon} \,\sigma(t, Z_{\epsilon}^j(t))\,d W(t), ~~ Z_{\epsilon}^j(0) = x_0, \notag \\
   ~(\mathcal{L}_j, \mathcal{L}_{\neg j}^{\ast}) \in \mathscr{L}, ~ j \in \mathcal{N}, \label{EQ44} 
 \end{align}
where $\sigma(t, Z_{\epsilon}^j(t)) \in \mathbb{R}^{d \times d}$ is a diffusion term, $W(t)$ is a $d$-dimensional Wiener process and $\epsilon$ is a small positive number, which represents the level of perturbation in the system. Note that we assume here there exists a set of equilibrium feedback operators $\bigl(\mathcal{L}_1^{\ast}, \mathcal{L}_2^{\ast}, \ldots, \mathcal{L}_N^{\ast} \bigr) \in \mathscr{L}$, when $\epsilon = 0$ (which corresponds to the unperturbed multi-channel system). Then, we investigate, as $\epsilon \rightarrow 0$, the asymptotic stability behavior of an invariant measure for the family of Frobenius-Perron operators $\Bigl\{\mathit{P}_{\epsilon, t}^{(\mathcal{L}_j, \mathcal{L}_{\neg j}^{\ast})} \Bigr\}_{t \ge 0}$, with $(\mathcal{L}_j, \mathcal{L}_{\neg j}^{\ast}) \in \mathscr{L}$, which corresponds to the multi-channel system with a small random perturbation.\footnote{We remark that such a solution for Equation~\eqref{EQ44} is assumed to have continuous sample paths with probability one (see Kunita \cite{Kun90} for additional information).}
\begin{remark}
Note that, in general, the evolution of the density function is given by
\begin{align*} 
 \mathit{P}_{\epsilon, t}^{(\mathcal{L}_j^{\ast}, \mathcal{L}_{\neg j}^{\ast})}\vartheta(x) = \int_{X} \Gamma_{\epsilon, t}(x, y)\vartheta(y)\mu(dy), \quad \forall t \ge 0,
\end{align*}
where $\Gamma_{\epsilon, t}(\cdot,\cdot)$ is the kernel (i.e., the fundamental solution), which is independent of the initial density function $\vartheta(x) \in D(X, \mathscr{A}, \mu)$. Moreover, it is well known that the solution, which is associated with Cauchy problem, satisfies the Fokker-Planck (or Kolmogorov forward) equation that is completely specified, with some additional regularity conditions, by $\Bigl(A(t) +  \sum\nolimits_{j \in \mathcal{N}} B_j(t) \mathcal{L}_j^{\ast}(t)\Bigr)Z_{\epsilon}(t)$ and $\sqrt{\epsilon}\,\sigma(t, Z_{\epsilon}(t))$ (e.g., see also \cite{GikSko75} or \cite{Kun90}).
\end{remark}

In what follows, we provide additional results, based on the asymptotic stability of an invariant measure, that partly establish the resilient behavior for the set of equilibrium feedback operators with respect to the random perturbation in the system. 

\begin{proposition}\label{PR3}
For any continuous density function $\vartheta(x) \in D(X, \mathscr{A}, \mu)$, suppose that
\begin{align}
 \sup_{(\mathcal{L}_j, \mathcal{L}_{\neg j}^{\ast}) \in \mathscr{L}}\,\Bigl \Vert \mathit{P}_{\epsilon, t}^{(\mathcal{L}_j, \mathcal{L}_{\neg j}^{\ast})} \vartheta(x) - \mathit{P}_t^{(\mathcal{L}_j^{\ast}, \mathcal{L}_{\neg j}^{\ast})} \vartheta(x) \Bigr\Vert_{L^1(X, \mathscr{A}, \mu)}, ~~ \forall t \ge 0, \forall j \in \mathcal{N}, \label{EQ45} 
\end{align}
tends to zero in a weak* topology on $X$ as $\epsilon \rightarrow 0$. Then, the weak limit of invariant measure $\mu_{\ast}^{\epsilon}$ of $\Big\{\mathit{P}_{\epsilon,t}^{(\mathcal{L}_j^{\ast}, \mathcal{L}_{\neg j}^{\ast})} \Bigr\}_{t \ge 0}$ is absolutely continuous with respect to the invariant measure $\mu_{\ast}$, where the latter corresponds to the family of Frobenius-Perron operators $\Bigl\{\mathit{P}_t^{(\mathcal{L}_j^{\ast}, \mathcal{L}_{\neg j}^{\ast})} \Bigr\}_{t \ge 0}$.
\end{proposition}

\begin{proof}
Note that, from the standard perturbation arguments for linear operators, if the set of equilibrium feedback operators $\bigl(\mathcal{L}_1^{\ast}, \mathcal{L}_2^{\ast}, \ldots, \mathcal{L}_N^{\ast} \bigr) \in \mathscr{L}$ and the fixed-point density function $\vartheta_{\ast}(x) \in D(X, \mathscr{A}, \mu)$ (i.e., $\mathit{P}_t^{(\mathcal{L}_j^{\ast}, \mathcal{L}_{\neg j}^{\ast})}\vartheta_{\ast}(x)=\vartheta_{\ast}(x)$, $\forall t \ge 0$) satisfy Proposition~\ref{PR2}. Then, the following holds
\begin{align*}
 \lim_{\epsilon \rightarrow 0} \left(\sup_{x \in X} \,\int_{X} \left \vert \mathit{P}_{\epsilon, t}^{(\mathcal{L}_j^{\ast}, \mathcal{L}_{\neg j}^{\ast})} \vartheta(x) - \mathit{P}_t^{(\mathcal{L}_j^{\ast}, \mathcal{L}_{\neg j}^{\ast})} \vartheta(x)\right\vert \mu(dx) \right)=0, ~~ \forall \vartheta(x) \in D(X, \mathscr{A}, \mu),
\end{align*}
for any fixed $t \ge 0$. This further implies the following
\begin{align*}
 \lim_{\epsilon \rightarrow 0} \,\Bigl \Vert \vartheta_{\ast}^{\epsilon}(x) - \vartheta_{\ast}(x) \Bigr\Vert_{L^1(X, \mathscr{A}, \mu)} = 0,
\end{align*}
where $\vartheta_{\ast}^{\epsilon}(x)$ is invariant of $\mathit{P}_{\epsilon, t}^{(\mathcal{L}_j^{\ast}, \mathcal{L}_{\neg j}^{\ast})}$ for each fixed $t \ge 0$.

In order for $\mu_{\ast}^{\epsilon}$ to be absolutely continuous with respect to $\mu_{\ast}$, i.e., $\mu_{\ast}^{\epsilon} \ll \mu_{\ast}$ and $\mu_{\ast}^{\epsilon} (A) = \int_{A} \vartheta_{\ast}^{\epsilon}(x) \mu_{\ast}^{\epsilon}(dx)$, $\forall A \in \mathscr{A}$, it is suffice to show that, for any fixed $t \ge 0$, the family of Frobenius-Perron operators $\mathit{P}_{\epsilon, t}^{(\mathcal{L}_j, \mathcal{L}_{\neg j}^{\ast})}$, with respect to $(\mathcal{L}_j, \mathcal{L}_{\neg j}^{\ast}) \in \mathscr{L}$ for all $j \in \mathcal{N}$, should not be too different from $\mathit{P}_t^{(\mathcal{L}_j^{\ast}, \mathcal{L}_{\neg j}^{\ast})}$ for small $\epsilon \ge 0$ (cf. Remark~\ref{RM8} below). 

On the other hand, under the game-theoretic framework (cf. Proposition~\ref{PR2}), each of these feedback operators are required to respond in some sense of best-response correspondence to the others feedback strategies in the system. As a result of this, the following will hold true
\begin{align*}
 \sup_{(\mathcal{L}_j, \mathcal{L}_{\neg j}^{\ast}) \in \mathscr{L}}\,\Bigl \Vert \mathit{P}_{\epsilon, t}^{(\mathcal{L}_j, \mathcal{L}_{\neg j}^{\ast})} \vartheta_{\ast}(x) -  \underbrace{\mathit{P}_t^{(\mathcal{L}_j^{\ast}, \mathcal{L}_{\neg j}^{\ast})}\vartheta_{\ast}(x)}_{=\vartheta_{\ast}(x), ~\forall t \ge 0} \Bigr\Vert_{L^1(X, \mathscr{A}, \mu)} \rightarrow 0 \quad  \text{as} \quad \epsilon \rightarrow 0,
\end{align*}
for each $j \in \mathcal{N}$, when only the set of feedback operators attains a robust/stable (game-theoretic) equilibrium solution $(\mathcal{L}_j^{\ast}, \mathcal{L}_{\neg j}^{\ast}) \in \mathscr{L}$. Note that, in the above equation, the supremum is computed with respect to $\mathcal{L}_j$ with $(\mathcal{L}_j, \mathcal{L}_{\neg j}^{\ast}) \in \mathscr{L}$ for each $j \in \mathcal{N}$, while others $\mathcal{L}_{\neg j}^{\ast}$ remain fixed, and when there is also a small random perturbation in the system.

Then, we see that $\mu_{\ast}^{\epsilon}$ tends to $\mu_{\ast}$ weakly as $\epsilon \rightarrow 0$. This completes the proof. \qed
\end{proof}

\begin{remark}\label{RM8}
We remark that, in general, the relation between $\mathit{P}_{\epsilon, t}^{(\mathcal{L}_j^{\ast}, \mathcal{L}_{\neg j}^{\ast})}$ and $\mathit{P}_t^{(\mathcal{L}_j^{\ast}, \mathcal{L}_{\neg j}^{\ast})}$ depends on the family of transformations $\Bigl\{\Phi_t^{(\mathcal{L}_j^{\ast}, \mathcal{L}_{\neg j}^{\ast})} \Bigr\}_{t \ge 0}$ (with respect to the set of equilibrium feedback operators $(\mathcal{L}_j^{\ast}, \mathcal{L}_{\neg j}^{\ast}) \in \mathscr{L}$) as well as on the measure space $L^1(X, \mathscr{A}, \mu)$ (see also \cite{BalYo93} and \cite{Kel82}).\end{remark}

We conclude this subsection with the following corollary, which is concerned with the resilient behavior of the set of equilibrium feedback operators, when there is a small random perturbation in the system. The proof follows similar arguments as in the proofs of Proposition~\ref{PR3} and Corollary~\ref{CR4}, and therefore will be omitted.
\begin{corollary}\label{CR5}
For $\epsilon > 0$ and $\operatorname{supp}\vartheta_{\ast}^{\epsilon}(x) \subset \operatorname{supp}\vartheta_{\ast}(x)$, if the relative entropy of the multi-channel system, with a random perturbation term, satisfies the following condition
\begin{align}
 \mathcal{H}_{\rm r}\Bigr(\mathit{P}_{\epsilon, t}^{(\mathcal{L}_j^{\ast}, \mathcal{L}_{\neg j}^{\ast})}\vartheta(x)\,\bigl\lvert\,\mathit{P}_t^{(\mathcal{L}_j^{\ast}, \mathcal{L}_{\neg j}^{\ast})}\vartheta(x)\Bigl)\le \theta_{\epsilon}^{(\mathcal{L}_j^{\ast}, \mathcal{L}_{\neg j}^{\ast})}, ~~ \forall t \ge 0, ~ \forall \vartheta(x) \in D(X, \mathscr{A}, \mu), \label{EQ46} 
\end{align}
where $\theta_{\epsilon}^{(\mathcal{L}_j^{\ast}, \mathcal{L}_{\neg j}^{\ast})}$ is a small positive number that depends on $\epsilon$ (and also tends to zero as $\epsilon \rightarrow 0$). Then, the set of equilibrium feedback operators $\bigl(\mathcal{L}_1^{\ast}, \mathcal{L}_2^{\ast}, \ldots, \mathcal{L}_N^{\ast} \bigr) \in \mathscr{L}$ exhibits a resilient behavior.
\end{corollary}
The above corollary states that the set of equilibrium feedback operators exhibits a resilient behavior, when the contribution of the perturbation term, to move away the system from the invariant measure $\mu_{\ast}$, is bounded from above for all $t \ge 0$.

We also note that the following holds true (see Equation~\eqref{EQ45})
\begin{align}
\lim_{t \rightarrow \infty} \, \Bigl \Vert \mathit{P}_{\epsilon, t}^{(\mathcal{L}_j^{\ast}, \mathcal{L}_{\neg j}^{\ast})} \vartheta(x) - \vartheta_{\ast}(x) \Bigr \Vert_{L^1(X, \mathscr{A}, \mu)} \rightarrow 0 \quad  \text{as} \quad \epsilon \rightarrow 0, \label{EQ47} 
\end{align}
for any $\vartheta(x) \in D(X, \mathscr{A}, \mu)$. Therefore, such a bound in Equation~\eqref{EQ46} is an immediate consequence of this fact.\footnote{For small $\epsilon \ge 0$, notice that 
\begin{align*}
 \lim_{t \rightarrow \infty} \int_{X} \mathit{P}_{\epsilon, t}^{(\mathcal{L}_j^{\ast}, \mathcal{L}_{\neg j}^{\ast})} \vartheta(x) \mu(dx) = \int_{X} \vartheta_{\ast}^{\epsilon}(x)\mu(dx), \quad \forall \vartheta(x) \in D(X, \mathscr{A}, \mu),
\end{align*}
when the stochastic semigroup $\mathit{P}_{\epsilon, t}^{(\mathcal{L}_j^{\ast}, \mathcal{L}_{\neg j}^{\ast})}$ is asymptotically stable for each fixed $t \ge 0$ (cf. \cite[Sec.~11.9]{LasMac94}).}

\begin{remark}
Finally, we note that although we have not discussed the limiting behavior, as $\epsilon \rightarrow 0$, of the family of measures $\bigl\{\mu_{\ast}^{\epsilon}\bigr\}$ on the space $L^1(X, \mathscr{A}, \mu)$. It appears that the {\em theory of large deviations} can be used to estimate explicitly the rate at which this family of measures converges to the limit measure $\mu_{\ast}$, where the latter is invariant with respect to $\Phi_t^{(\mathcal{L}_j^{\ast}, \mathcal{L}_{\neg j}^{\ast})}$ for each fixed $t \ge 0$ (e.g., see \cite{Tou09}, \cite{Eli95} or \cite{DemZe98} for a detailed exposition of this theory).
\end{remark}

\end{document}